\documentclass{amsart}

\title[Open class determinacy is preserved by forcing]{Open class determinacy is preserved by forcing}

\author{Joel David Hamkins}
 \address[Joel David Hamkins]
         {Professor of Logic, University of Oxford, and Sir Peter Strawson Fellow in Philosophy, University College, Oxford}
\email{jhamkins@gc.cuny.edu}
\urladdr{http://jdh.hamkins.org}

\author{W.~Hugh Woodin}
 \address[W.~Hugh Woodin]
        {Professor of Philosophy and of Mathematics, Department of Philosophy, Emerson Hall 209A, Harvard University,
25 Quincy Street Cambridge, MA 02138}
 \email{woodin@math.harvard.edu}

\thanks{Thanks to Victoria Gitman for helpful comments. Commentary can be made about this article on the first author's blog at \href{http://jdh.hamkins.org/open-class-determinacy-preserved-by-forcing}{http://jdh.hamkins.org/open-class-determinacy-preserved-by-forcing}.}

\usepackage{latexsym,amsfonts,amsmath,amssymb,mathrsfs}
\usepackage{bbm}
\usepackage[hidelinks]{hyperref}
\usepackage{relsize}
\usepackage[rgb,dvipsnames]{xcolor}
\usepackage{tikz}
\usetikzlibrary{arrows,arrows.meta,petri,topaths,positioning,shapes,shapes.misc,patterns,calc,decorations.pathreplacing,hobby}
\usepackage{enumitem}

\newtheorem{theorem}{Theorem}

\newtheorem*{maintheorem*}{Main Theorem}
\newtheorem*{maintheorems*}{Main Theorems}
\newtheorem{corollary}[theorem]{Corollary}
\newtheorem*{corollary*}{Corollary}
\newtheorem*{corollaries*}{Corollaries}

\newtheorem{question}[theorem]{Question}
\newtheorem*{question*}{Question}

\newtheorem*{questions*}{Questions}
\newtheorem*{mainquestion*}{Main Question} 
\newtheorem*{openquestion*}{Open Question} 
\newtheorem{observation}[theorem]{Observation}

\newcommand{\QED}{\end{proof}}

\def\proclaim[#1]{{\bf #1}}
\def\BF#1.{{\bf #1.}}

\def\says#1:#2\par{\item[#1] #2\par}

\newcommand{\Godel}{G\"odel}

\renewcommand{\P}{{\mathbb P}}

\newcommand{\one}{\mathbbm{1}} 
\makeatletter
\newcommand{\dotminus}{\mathbin{\text{\@dotminus}}}
\newcommand{\@dotminus}{%
  \ooalign{\hidewidth\raise1ex\hbox{.}\hidewidth\cr$\m@th-$\cr}%
}
\makeatother

\newcommand{\of}{\subseteq}

\newcommand{\set}[1]{\{\,{#1}\,\}}

\newcommand{\Con}{\mathop{{\rm Con}}}

\newcommand{\restrict}{\upharpoonright} 
\newcommand{\satisfies}{\models}
\newcommand{\forces}{\Vdash}

\newcommand{\smalllt}{\mathrel{\mathchoice{\raise2pt\hbox{$\scriptstyle<$}}{\raise1pt\hbox{$\scriptstyle<$}}{\raise0pt\hbox{$\scriptscriptstyle<$}}{\scriptscriptstyle<}}}
\newcommand{\smallleq}{\mathrel{\mathchoice{\raise2pt\hbox{$\scriptstyle\leq$}}{\raise1pt\hbox{$\scriptstyle\leq$}}{\raise1pt\hbox{$\scriptscriptstyle\leq$}}{\scriptscriptstyle\leq}}}

\newcommand{\boolval}[1]{\mathopen{\lbrack\!\lbrack}\,#1\,\mathclose{\rbrack\!\rbrack}}
\def\[#1]{\boolval{#1}}
\newbox\gnBoxA
\newdimen\gnCornerHgt
\setbox\gnBoxA=\hbox{\tiny$\ulcorner$}
\global\gnCornerHgt=\ht\gnBoxA
\newdimen\gnArgHgt
\def\gcode #1{%
\setbox\gnBoxA=\hbox{$#1$}%
\gnArgHgt=\ht\gnBoxA%
\ifnum     \gnArgHgt<\gnCornerHgt \gnArgHgt=0pt%
\else \advance \gnArgHgt by -\gnCornerHgt%
\fi \raise\gnArgHgt\hbox{\tiny$\ulcorner$} \box\gnBoxA %
\raise\gnArgHgt\hbox{\tiny$\urcorner$}}
\newcommand{\UnderTilde}[1]{{\setbox1=\hbox{$#1$}\baselineskip=0pt\vtop{\hbox{$#1$}\hbox to\wd1{\hfil$\sim$\hfil}}}{}}
\newcommand{\Undertilde}[1]{{\setbox1=\hbox{$#1$}\baselineskip=0pt\vtop{\hbox{$#1$}\hbox to\wd1{\hfil$\scriptstyle\sim$\hfil}}}{}}
\newcommand{\undertilde}[1]{{\setbox1=\hbox{$#1$}\baselineskip=0pt\vtop{\hbox{$#1$}\hbox to\wd1{\hfil$\scriptscriptstyle\sim$\hfil}}}{}}
\newcommand{\UnderdTilde}[1]{{\setbox1=\hbox{$#1$}\baselineskip=0pt\vtop{\hbox{$#1$}\hbox to\wd1{\hfil$\approx$\hfil}}}{}}
\newcommand{\Underdtilde}[1]{{\setbox1=\hbox{$#1$}\baselineskip=0pt\vtop{\hbox{$#1$}\hbox to\wd1{\hfil\scriptsize$\approx$\hfil}}}{}}

\def\<#1>{\left\langle#1\right\rangle}

\newcommand{\Ord}{\mathord{{\rm Ord}}}

\newcommand{\ETR}{{\rm ETR}}
\newcommand\ETRord{\ETR_{\Ord}}

\newcommand{\ZFC}{{\rm ZFC}}

\newcommand{\KM}{{\rm KM}}
\newcommand{\GB}{{\rm GB}}
\newcommand{\GBC}{{\rm GBC}}

\newcommand{\cell}[1]{\boxit{\hbox to 17pt{\strut\hfil$#1$\hfil}}}
\newcommand{\head}[2]{\lower2pt\vbox{\hbox{\strut\footnotesize\it\hskip3pt#2}\boxit{\cell#1}}}
\newcommand{\boxit}[1]{\setbox4=\hbox{\kern2pt#1\kern2pt}\hbox{\vrule\vbox{\hrule\kern2pt\box4\kern2pt\hrule}\vrule}}
\newcommand{\Col}[3]{\hbox{\vbox{\baselineskip=0pt\parskip=0pt\cell#1\cell#2\cell#3}}}
\newcommand{\tapenames}{\raise 5pt\vbox to .7in{\hbox to .8in{\it\hfill input: \strut}\vfill\hbox to
.8in{\it\hfill scratch: \strut}\vfill\hbox to .8in{\it\hfill output: \strut}}}
\newcommand{\Head}[4]{\lower2pt\vbox{\hbox to25pt{\strut\footnotesize\it\hfill#4\hfill}\boxit{\Col#1#2#3}}}
\newcommand{\Dots}{\raise 5pt\vbox to .7in{\hbox{\ $\cdots$\strut}\vfill\hbox{\ $\cdots$\strut}\vfill\hbox{\
$\cdots$\strut}}}
\usepackage[backend=bibtex,style=alphabetic,maxbibnames=15,maxcitenames=6,dateabbrev=false]{biblatex}
\renewcommand{\UrlFont}{\sffamily\smaller} 
\renewbibmacro{in:}{\ifentrytype{article}{}{\printtext{\bibstring{in}\intitlepunct}}} 
\DeclareFieldFormat{url}{{\relsize{-2}\UrlFont\url{#1}}} 
\DeclareFieldFormat{urldate}{
  (version \thefield{urlday}\addspace%
  \mkbibmonth{\thefield{urlmonth}}\addspace%
  \thefield{urlyear}\isdot)}
\DeclareFieldFormat{eprint:arxiv}{
  \ifhyperref
    {\href{http://arxiv.org/abs/#1}{%
        arXiv\addcolon\nolinkurl{#1}}\iffieldundef{eprintclass}{}{\UrlFont{\mkbibbrackets{\thefield{eprintclass}}}}}
    {arXiv\addcolon\nolinkurl{#1}\iffieldundef{eprintclass}{}{\UrlFont{\mkbibbrackets{\thefield{eprintclass}}}}}}
\bibliography{HamkinsBiblio,MathBiblio,WebPosts}

\begin{document}

\begin{abstract}
The principle of open class determinacy is preserved by pre-tame class forcing, and after such forcing, every new class well-order is isomorphic to a ground-model class well-order. Similarly, the principle of elementary transfinite recursion $\ETR_\Gamma$ for a fixed class well-order $\Gamma$ is preserved by pre-tame class forcing. The full principle \ETR\ itself is preserved by countably strategically closed pre-tame class forcing, and after such forcing, every new class well-order is isomorphic to a ground-model class well-order. Meanwhile, it remains open whether \ETR\ is preserved by all forcing, including the forcing merely to add a Cohen real.
\end{abstract}

\maketitle

\section{Introduction}\label{Section.Introduction}

The principle of \emph{elementary transfinite recursion} \ETR---according to which every first-order class recursion along any well-founded class relation has a solution---has emerged as a central organizing concept in the hierarchy of second-order set theories from \Godel-Bernays set theory \GBC\ up to Kelley-Morse set theory \KM\ and beyond. Many of the principles in the hierarchy can be seen as asserting that certain class recursions have solutions.%
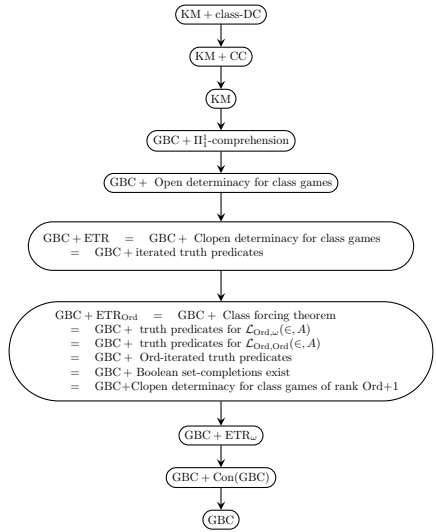
\begin{figure}[h]\label{Figure.Theories}
  \begin{tikzpicture}[theory/.style={draw,thin,rounded rectangle,scale=.45,minimum height=5.6mm},scale=.28]
     \draw (0:0) node[theory] (ZFC) {$\GBC$}
           ++(90:2) node[theory] (ConZFC) {$\GBC+\Con(\GBC)$}
           ++(90:2) node[theory] (ETRomega) {$\GBC+\ETR_\omega$}
           ++(90:4) node[theory] (ETRord) {\parbox[c][2.7cm]{10.5cm}{{\quad}$\GBC+\ETRord\quad=\quad\GBC+\text{ Class forcing theorem}$\\
            ${\qquad}=\quad\GBC+\text{ truth predicates for }\mathcal{L}_{\Ord,\omega}(\in,A)$\\
            ${\qquad}=\quad\GBC+\text{ truth predicates for }\mathcal{L}_{\Ord,\Ord}(\in,A)$\\
            ${\qquad}=\quad\GBC+\text{ $\Ord$-iterated truth predicates }$\\
            ${\qquad}=\quad\GBC\,+$ Boolean set-completions exist\quad \\
            ${\qquad}=\quad\GBC+\text{Clopen determinacy for class games of rank }\Ord+1$}}
           ++(90:5) node[theory] (ETR) {\parbox[c][1.2cm]{10.5cm}{$\GBC+\ETR\quad=\quad\GBC+\text{ Clopen determinacy for class games}$\\
           ${\qquad}=\quad\GBC+\text{iterated truth predicates}$}}
           ++(90:3) node[theory] (Open) {$\GBC+\text{ Open determinacy for class games}$}
           ++(90:2) node[theory] (Pi11) {$\GBC+\Pi^1_1$-comprehension}
           ++(90:2) node[theory] (KM) {$\KM$}
           ++(90:2) node[theory] (KM+CC) {$\KM+\text{CC}$}
           ++(90:2) node[theory] (KM+DC) {$\KM+\text{class-DC}$};
     \draw[<-,>=stealth]
                        (ZFC) edge (ConZFC)
                        (ConZFC) edge (ETRomega)
                        (ETRomega) edge (ETRord)
                        (ETRord) edge (ETR)
                        (ETR) edge (Open)
                        (Open) edge (Pi11)
                        (Pi11) edge (KM)
                        (KM) edge (KM+CC)
                        (KM+CC) edge (KM+DC);
  \end{tikzpicture}
\end{figure}

In addition, many of these principles, including \ETR\ and its variants, are equivalently characterized as determinacy principles for certain kinds of class games. Thus, the hierarchy is also fruitfully unified and organized by determinacy ideas.

This hierarchy of theories is the main focus of study in the reverse mathematics of second-order set theory, an emerging subject aiming to discover the precise axiomatic principles required for various arguments and results in second-order set theory. The principle \ETR\ itself, for example, is equivalent over \GBC\ to the principle of clopen determinacy for class games~\cite{GitmanHamkins2016:OpenDeterminacyForClassGames} and also to the existence of iterated elementary truth predicates~\cite{GitmanHamkins2016:OpenDeterminacyForClassGames}, \cite{Fujimoto2012:Classes-and-truths-in-set-theory}; since every clopen game is also an open game, the principle \ETR\ is naturally strengthened by the principle of open determinacy for class games, and this is a strictly stronger principle~\cite{GitmanHamkins2016:OpenDeterminacyForClassGames}, \cite{Hachtman:Determinacy-separations-for-class-games}, \cite{Sato:Inductive-dichotomy-separation-of-open-and-clopen-class-determinacies}; the weaker principle $\ETRord$, meanwhile, asserting solutions to class recursions of length $\Ord$, is equivalent to the class forcing theorem, which asserts that every class forcing notion admits a forcing relation, to the existence of set-complete Boolean completions of any class partial order, to the existence of $\Ord$-iterated elementary truth predicates, to the determinacy of clopen games of rank at most $\Ord+1$, and to other natural set-theoretic principles~\cite{GitmanHamkinsHolySchlichtWilliams:The-exact-strength-of-the-class-forcing-theorem}.

Since one naturally seeks in any subject to understand how one's fundamental principles and tools interact, we should like in this article to consider how these second-order set-theoretic principles are affected by forcing. These questions originated in previous work of Gitman and Hamkins, and question~\ref{Question.ETR-preserved?} also appears in the dissertation of Kameryn Williams~\cite[question~1.36]{Williams2018:dissertation}, which was focused on the structure of models of second-order set theories.

It is well-known, of course, that \ZFC, \GBC, and \KM\ are preserved by set forcing and by tame class forcing, and this is true for other theories in the hierarchy, such as $\GBC+\Pi^1_n$-comprehension and higher levels of the second-order comprehension axiom. The corresponding forcing preservation theorem for \ETR\ and for open class determinacy, however, has not been known.

\begin{question}\label{Question.ETR-preserved?}
Is \ETR\ preserved by forcing?
\end{question}

\begin{question}\label{Question.Open-determinacy-preserved?}
Is open class determinacy preserved by forcing?
\end{question}

We intend to ask in each case about class forcing as well as set forcing. Question~\ref{Question.ETR-preserved?} is closely connected with the question of whether forcing can create new class well-order order types, longer than any class well-order in the ground model. Specifically, Gitman and Hamkins had observed that $\ETR_\Gamma$ for a specific class well-order $\Gamma$ is preserved by pre-tame class forcing (see theorem~\ref{Theorem.ETR_Gamma-preserved}, which is the same as statement~2 in the main theorem), and they noted that this would imply that the full principle \ETR\ would also be preserved, if no fundamentally new class well-orders are created by the forcing. In light of the fact that forcing over models of \ZFC\ adds no new ordinals, that would seem reasonable, but proof is elusive, and the question remains open. Can forcing add new class well-orders, perhaps very tall ones that are not isomorphic to any ground model class well-order? Perhaps there are some very strange models of \GBC\ that gain new class well-order order types in a forcing extension.

\begin{question}\label{Question.Forcing-new-class-well-orders?}
 Assume \GBC. After forcing, must every new class well-order be isomorphic to a ground-model class well-order? Does \ETR\ imply this?
\end{question}

Our main theorem provides a full affirmative answer to question~\ref{Question.Open-determinacy-preserved?}, and partial affirmative answers to questions~\ref{Question.ETR-preserved?} and~\ref{Question.Forcing-new-class-well-orders?}.

\medskip\goodbreak

\begin{maintheorem*}\
  \begin{enumerate}
    \item Open class determinacy is preserved by pre-tame class forcing. After such forcing, every new class well-order is isomorphic to a ground-model class well-order.
    \item The principle $\ETR_\Gamma$, for any fixed class well order $\Gamma$, is preserved by pre-tame class forcing.
    \item The full principle \ETR\ is preserved by countably strategically closed pre-tame class forcing. After such forcing, every new class well-order is isomorphic to a ground-model class well-order.
  \end{enumerate}
\end{maintheorem*}

We should like specifically to highlight the fact that questions~\ref{Question.ETR-preserved?} and~\ref{Question.Forcing-new-class-well-orders?} remain open even in the case of the forcing to add a Cohen real. Is \ETR\ preserved by the forcing to add a Cohen real? After adding a Cohen real, is every new class well-order isomorphic to a ground-model class well-order? One naturally expects affirmative answers, especially in a model of \ETR.

\section{Background on \ETR\ and open class determinacy}

Let us briefly review some background material we shall require. The principle of \emph{elementary transfinite recursion} \ETR\ asserts that one may undertake class recursions along any class well order. Specifically, for any class well order $\<\Gamma,\leq_\Gamma>$ and first-order formula $\varphi$ with class parameter $Z$, the principle \ETR\ asserts that there is a solution of the recursion, that is, a class $S\of\Gamma\times V$ such that
 $$S_\gamma=\set{x\mid \varphi(x,S\restrict\gamma,Z)}.$$
Thus, the sections $S_\gamma=\set{x\mid (\gamma,x)\in S}$ of the solution class are defined by recursion on $\gamma\in\Gamma$ using the formula $\varphi$ and making reference to the earlier sections $S\restrict\gamma=\set{(\alpha,x)\in S\mid \alpha<_\Gamma\gamma}$. For a fixed class well order $\Gamma$, the principle $\ETR_\Gamma$ asserts that all such recursions of length $\Gamma$ have solutions. The principle $\ETR_\omega$, for example, already has nontrivial strength over \GBC, because the Tarskian recursive definition of truth is precisely such a class recursion of length $\omega$, defining the satisfaction relation by recursion on formulas. See~\cite{GitmanHamkins2016:OpenDeterminacyForClassGames} for a full account of \ETR\ and related matters.

The principle of \emph{open determinacy for class games} asserts that every open class game has a winning strategy for one of the players. Specifically, for any class $X$ and any class $A\of X^\omega$, we consider the two player game in which the players take turns to construct a sequence $\<a_n\mid n\in\omega>$.
$$\begin{array}{rccccccccccc}
{\rm I}\quad   & a_0   &                   & a_2   &                   & a_4   & \\
{\rm II}\quad  &       & a_1    &       & a_3    &       & a_5\quad\raise 8pt\hbox{$\cdots$}\\
\end{array}$$
Player I wins a play of the game if the sequence $\<a_n\mid n\in\omega>$ is in the payoff class $A$, and otherwise player II wins. The game is \emph{determined}, if one of the player has a winning strategy, which is a class function from partial plays to the next move, such that any play played in accordance with it yields a win for that player. The game is \emph{open} for a player, if all winning plays for that player are essentially won at a finite stage of play, in the sense that there is some initial segment of the play such that all extensions of that finite play are winning for that player. This is equivalent to saying that the payoff class for that player is open in the product topology on $X^\omega$. A game is \emph{clopen}, if it is open for each player. If one regards the game as over when play has reached a finite sequence all of whose extensions have the same outcome (in other words, the interior of that player's payoff class), then the clopen games are characterized as those whose game tree, the tree of all possible plays in the game, is well-founded.

Gitman and Hamkins~\cite{GitmanHamkins2016:OpenDeterminacyForClassGames} proved that the principle of clopen determinacy for class games is equivalent to the principle \ETR\ of elementary transfinite recursion. Hachtman~\cite{Hachtman:Determinacy-separations-for-class-games} proved that open determinacy for class games is a strictly stronger principle, and Sato~\cite{Sato:Inductive-dichotomy-separation-of-open-and-clopen-class-determinacies} proved that it is stronger in consistency strength.

A forcing notion $\P$ is \emph{countably strategically closed}, if player II has a winning strategy in the game in which the players play a descending sequence $p_0\geq p_1\geq p_2\geq\cdots$ of conditions from $\P$, with player II winning if there is a condition $p\in\P$ below every $p_n$.
$$\begin{array}{rccccccccccc}
{\rm I}\quad   & p_0   &                   & p_2   &                   & p_4   & \\
{\rm II}\quad  &       & p_1    &       & p_3    &       & p_5\quad\raise 8pt\hbox{$\cdots$}\\
\end{array}$$
Every countably closed forcing notion is countably strategically closed, and every countably strategically closed forcing notion is $\leq\omega$-distributive, meaning that it adds no new $\omega$-sequences over the ground model; one can use the turns of player I to decide more and more of any desired name for an $\omega$-sequence, and the strategy for player II produces a limit condition deciding the entire sequence of values.

In this article, we consider the theory $\GBC^-$, which is \Godel-Bernays set theory without the power set axiom and with the global choice axiom in the form of a bijection of the universe with $\Ord$; over $\GB^-$, this form of global choice is strictly stronger than the assertion merely that one has a class well-ordering of the universe, although the two forms are equivalent in \GB.

\section{Background on class well orders and well-founded relations}

Let us next develop a little of the background theory of well-founded class relations and their rankings by class well-orders. A binary class relation $\vartriangleleft$ on a class $X$ is \emph{well-founded} if every nonempty subset $x\of X$ has a $\vartriangleleft$-least member. This is equivalent over $\GBC^-$ to the assertion that every nonempty subclass $Y\of X$ has a $\vartriangleleft$-minimal element; and it is also equivalent to the assertion that there is no $\vartriangleleft$-descending sequence $\<a_n\mid n\in\omega>$, that is, for which $a_{n+1}\mathrel{\vartriangleleft}a_n$. Thus, well-foundedness is a first-order concept in second-order set theory, unlike the situation in second-order arithmetic, where it is $\Pi^1_1$-complete. This difference is the source of certain disanalogies between second-order set theory and the arithmetic counterpart, despite the generally robust positive analogy between these fields in many other respects.

A \emph{class well-order} is a class linear order relation that is well-founded. We say that a binary relation $\vartriangleleft$ is \emph{graded} or \emph{ranked} by a class well-order $\<\Gamma,\leq_\Gamma>$, if there is a map $\pi$ from the field of $\vartriangleleft$ to $\Gamma$ such that $a\vartriangleleft b$ implies $\pi(a)<_\Gamma \pi(b)$. The existence of such a ranking easily implies that $\vartriangleleft$ is well-founded. Such a ranking $\pi$ is \emph{continuous}, if for every $b\in X$ the object $\pi(b)$ is the $\leq_\Gamma$-least element of $\Gamma$ that is strictly above $\pi(a)$ for all $a\vartriangleleft b$ (note that we impose this requirement not merely at limits, but also at successors, and so this is not strictly a topological notion). Note that the property of $\pi$ being a continuous ranking of $\<X,\vartriangleleft>$ into $\<\Gamma,\leq_\Gamma>$ is a first-order property of these classes, and furthermore, the continuous ranking is unique when it exists, since there can be no least point of difference between two of them.

\begin{theorem}\label{Theorem.Well-founded-relations-are-ranked}
Assume $\GBC^-$. Every well-founded class relation $\<X,\vartriangleleft>$ admits a ranking to some class well-order $\<\Gamma,\leq_\Gamma>$.
\end{theorem}

\begin{proof}
Fix any well-founded class relation $\<X,\vartriangleleft>$, and let $\Gamma$ consist of the class of finite $\vartriangleleft$-descending sequences in $X$. This is a well-founded tree under extension, growing downward so that longer sequences are lower in the order. Let $\leq_\Gamma$ be the Kleene-Brouwer order on this tree, by which one sequence is $\leq_\Gamma$-below another if it extends it or if at the place of first difference, it branches to the left, using a fixed global well-order. This is a linear order, since any two distinct finite descending sequences must have a first place where they differ; and it is a well-order, since any infinite descending sequence in the Kleene-Brouwer order would give rise either to an infinite descending sequence in $\vartriangleleft$, which doesn't exist, or to an infinite descending sequence in the global well-order, which also is impossible.

To define the ranking, let us assign each element $a\in X$ to the $\Gamma$-least sequence ending with $a$. It follows easily that $b\vartriangleleft a$ implies $\pi(b)<_\Gamma\pi(a)$, since any sequence ending with $a$ can be extended to a sequence ending with $b$, which is therefore lower in the Kleene-Brouwer order. Thus, it is a ranking of $\vartriangleleft$ by $\Gamma$.
\end{proof}

Theorem~\ref{Theorem.Well-founded-relations-are-ranked} shows that one does not need the \ETR\ principle in order to find rankings of a well-founded class relation, if one does not insist on the continuity requirement. But meanwhile, the rankings provided by the proof of theorem~\ref{Theorem.Well-founded-relations-are-ranked} are usually not continuous. One can see this easily from the fact that the ranking $\pi$ provided in the proof is injective, and therefore cannot be continuous on well-founded relations having distinct individuals with the same sets of predecessors or which should otherwise have the same rank, a situation that can occur even with finite well-founded relations.

\begin{theorem}\label{Theorem.ETR-implies-continuous-rankings}
 Assume $\GBC^-+\ETR$. Every well-founded class relation has a continuous ranking to some class well-order. Indeed, every ranking of a well-founded class relation by a class well-order can be refined to a continuous ranking.
\end{theorem}

\begin{proof}
Theorem~\ref{Theorem.Well-founded-relations-are-ranked} shows that every well-founded class relation $\<X,\vartriangleleft>$ admits a ranking $\pi$ by some class well-order $\<\Gamma,\leq_\Gamma>$. What we mean by the second statement is that for any such ranking, there is a continuous ranking $\rho$ of $\vartriangleleft$ by $\Gamma$ with the additional property that $\rho(x)\leq\pi(x)$ for every $x\in X$.

To see this, assume $\GBC^-+\ETR$, and suppose that $\pi$ is a ranking of $\vartriangleleft$ by $\Gamma$. By recursion on $\vartriangleleft$, define that $\rho(a)$ is the $\Gamma$-least element strictly above $\rho(b)$ for all $b\vartriangleleft a$, if there is one, and otherwise some default value. In fact, the default values are never needed,  as $\rho(a)\leq_\Gamma \pi(a)$ for all $a\in X$ by induction: if this is true for $b\vartriangleleft a$, then it must hold also for $a$ itself, for otherwise $\pi(a)$ would be a smaller strict upper bound of  $\rho(b)$ for $b\vartriangleleft a$, contrary to the definition. By construction, the ranking $\rho$ is continuous.
\end{proof}

In~\cite{GitmanHamkinsHolySchlichtWilliams:The-exact-strength-of-the-class-forcing-theorem}, it is proved in \GBC\ that every well-founded relation with an $\Ord+1$ ranking admits a continuous $\Ord+1$-ranking. So one doesn't need \ETR\ to establish the existence of continuous rankings, when the ranks do not exceed $\Ord$. But in the general case, it is not known how to produce continuous rankings except by using \ETR\ and defining them recursively.

\medskip\goodbreak

\begin{theorem}
 Assume $\GBC^-$. The principle \ETR\ implies the class well-order comparability principle: given any two class well-orders, one of them is isomorphic to an initial segment of the other. Indeed, this conclusion follows if merely every well-founded class relation has a continuous ranking.
\end{theorem}

\begin{proof}
Let's first give the direct argument from \ETR. Assume $\GBC^-+\ETR$ and suppose that $\<\Gamma,\leq_\Gamma>$ and $\<\Lambda,\leq_\Lambda>$ are two class well-order relations. Using $\ETR_\Gamma$, we may recursively define $\pi(\gamma)$ to be the $\Lambda$-least element $\lambda\in\Lambda$ not in the class $\set{\pi(\alpha)\mid\alpha\leq_\Gamma\gamma}$, if any, and otherwise the recursion is ended at this $\gamma$. If $\pi$ is defined on all of $\Gamma$, then $\pi$ is an isomorphism of $\Gamma$ with an initial segment of $\Lambda$, and if the construction ends at $\gamma$, then $\pi$ is an isomorphism of the initial segment $\Gamma\restrict\gamma$ with $\Lambda$.

Now let us give the argument assuming instead $\GBC^-$ plus the assertion that every well-founded class relation has a continuous ranking. Suppose that $\Gamma$ and $\Lambda$ are two class well-orders. Let $\Lambda\oplus\Gamma$ be the well-founded relation consisting of side-by-side copies of $\Lambda$ and $\Gamma$. By assumption, there is a continuous ranking $\pi:\Lambda\oplus\Gamma\to\Theta$ using some class well-order $\Theta$. It follows that $\pi\restrict\Lambda$ and $\pi\restrict\Gamma$ are in each case an isomorphism of their domain with an initial segment of $\Theta$. Since the initial segments of $\Theta$ are comparable, one of them is shorter or the same size as the other. By following the isomorphism to the shorter of these with the inverse of the other isomorphism, we thereby obtain an isomorphism of one of $\Lambda$ or $\Gamma$ to an initial segment of the other.
\end{proof}

The main open question concerning comparability is whether these principles might be equivalent. Let us also introduce the \emph{weak class well order comparability principle}, which asserts of any two class well orders merely that one of them is isomorphic to a suborder of the other (rather than specifically to an initial segment).

\begin{question}
Are any or all of the following statements equivalent over $\GBC^-$?
 \begin{enumerate}
   \item The principle \ETR\ of elementary transfinite recursion.
   \item Every class well-founded relation admits a continuous ranking.
   \item The class well order comparability principle.
   \item The weak class well order comparability principle.
 \end{enumerate}
\end{question}

We have proved above that $(1)\to(2)\to (3)$ and clearly also $(3)\to (4)$.

\begin{observation}
The class well-order comparability theorem is equivalent over $\GBC^-$ to the instances of comparability of a class well-order with its suborders, and more specifically to the assertion that every suborder of a class well-order is isomorphic to an initial segment of that order.
\end{observation}

\begin{proof}
Clearly the well-order comparability theorem implies those instances. Suppose we have two class well-orders $\Gamma$ and $\Lambda$. Let us compare $\Lambda$ with $\Gamma+\Lambda$, viewing the first as a suborder of the second. By assumption, these are comparable. It is easy to see that $\Gamma+\Lambda$ cannot be isomorphic to a proper initial segment of $\Lambda$, and so we must have an isomorphism of $\Lambda$ with an initial segment of $\Gamma+\Lambda$. If the range of this isomorphism is contained in the initial copy of $\Gamma$ in this order, then $\Lambda$ is isomorphic to an initial segment of $\Gamma$, and otherwise, by considering the inverse map, $\Gamma$ is isomorphic to an initial segment of $\Lambda$.
\end{proof}

In models of \ZFC, every well-founded set relation can be ranked by an ordinal, and for this reason, well-foundedness is absolute between any model of \ZFC\ and its inner models. A similar fact is true for \GBC\ models:

\begin{theorem}
 If $\<M,\mathcal{X}>$ and $\<N,\mathcal{Y}>$ are models of $\GBC^-$ with $M\of N$ and $\mathcal{X}\of\mathcal{Y}$, and where the ordinals of $M$ are the same as $N$, then the well-foundedness of relations $\vartriangleleft$ in $\mathcal{X}$ is absolute between $\<M,\mathcal{X}>$ and $\<N,\mathcal{Y}>$.
\end{theorem}

\begin{proof} If $\vartriangleleft$ is well-founded in the larger model $\<N,\mathcal{Y}>$, then it must be well-founded in the smaller model $\<M,\mathcal{X}>$, since every set in $M$ will have a $\vartriangleleft$-minimal element in $N$ and the same element works in $M$. Conversely, suppose that $\vartriangleleft$ is well-founded in the smaller model $\<M,\mathcal{X}>$. The hypotheses ensure, we claim, that every set in $N$ that is a subset of $M$ is covered by a set that is an element of $M$. To see this, note that if $a\in N$ and $a\of M$, then $a$ must be bounded in the global well-order of $M$, for otherwise it would provide a set-sized cofinal set of ordinals in $N$, contrary to $\GBC^-$. So $a\of b$ for some $b\in M$. Since $\vartriangleleft\restrict b$ is a well-founded set relation in $M$, it has a ranking to the ordinals of $M$, and this ranking still exists in $N$. So there is a $\vartriangleleft$-minimal element of $a$. Thus, the relation remains well-founded in $\<N,\mathcal{Y}>$, as desired.
\end{proof}

In particular, the well-foundedness of class relations is absolute between any model of $\GBC^-$ and its pre-tame class forcing extensions.

Meanwhile, one should be aware that a transitive model $\<M,\mathcal{X}>$, meaning that $M$ is a transitive set, might be wrong in its judgement of well-foundedness of a class relation $\vartriangleleft\in\mathcal{X}$. The model $\<M,\mathcal{X}>$ is a \emph{$\beta$-model}, if every binary class relation $\vartriangleleft$ that $\<M,\mathcal{X}>$ thinks is well-founded is actually well-founded. The fact is that every countable transitive model of \GBC\ is \GBC-realizable by classes that make it a non-$\beta$-model (e.g. see \cite{Williams2018:dissertation,Williams:least-models}). Note also that strange things can happen even with \ZFC\ models: results in~\cite{HamkinsYang:SatisfactionIsNotAbsolute} provide models of \ZFC, with the same ordinals and a relation in common, in fact a c.e.~relation on the natural numbers of the models, such that one of the models thinks that the relation is well-founded and the other does not.

In the case of set forcing over of a model of \ZFC, we know that no new ordinals are added. What is the situation for models of \GBC\ and class well-orders? In question~\ref{Question.Forcing-new-class-well-orders?}, we asked whether, after forcing over a model of \GBC, is every new class well-order isomorphic to a ground-model class well-order? If one imagines that the problematic issue would be new very tall orders, then it is also natural to ask, is every new class well-order ranked by a ground-model class well-order? For countably strategically closed forcing, the answer is affirmative.

\begin{theorem}\label{Theorem.Closed-forcing-every-class-wf-relation-is-ranked}
 Assume $\GBC^-$. Then after any countably strategically closed pre-tame class forcing, every well-founded class relation in the forcing extension is ranked there by a ground-model class well-order.
\end{theorem}

\begin{proof}
Suppose that $\P$ is a countably strategically closed pre-tame class forcing notion, witnessed by strategy $\eta$ for player II in the strategic closure game. Suppose that $\one\forces\dot\vartriangleleft$ is a well-founded class relation on a class $\dot X$.

Consider pairs of the form $(p,\dot a)$, where $p\in\P$ and $p\forces\dot a\in \dot X$. Let $Y$ be the class of all finite sequences of such pairs $\<(p_0,\dot a_0),\ldots,(p_n,\dot a_n)>$, where the conditions are descending $p_0\geq\cdots\geq p_n$, and not only descending but furthermore each $p_{k+1}$ is below the response of $\eta$ to the earlier moves in the strategic closure game of $\P$, for $k<n$ and in addition $p_{k+1}\forces \dot a_{k+1}\mathrel{\dot\vartriangleleft}\dot a_k$. Consider $Y$ as a tree (growing downward) under extension of these sequences.

Observe that $Y$ is well-founded, since if $\<(p_n,\dot a_n)\mid n\in\omega>$ is an infinite descending sequence, then because the conditions conform with $\eta$ it follows that there is a condition $q\leq p_n$ for all $n$, and this condition will force that the $\dot a_n$ form an infinite $\dot\vartriangleleft$-descending sequence, contradicting the assumption that it was well-founded.

Let $\preccurlyeq$ be the Kleene-Brouwer order on $Y$. This is a class well-order in the ground model. In the forcing extension $V[G]$, where $G\of\P$ is $V$-generic, we shall define a ranking of $\vartriangleleft=(\dot\vartriangleleft)_G$ into $\<Y,\preccurlyeq>$. For any $a\in X=(\dot X)_G$, pick a name $\dot a$ with $(\dot a)_G=a$ and let $\pi(a)$ be the $\preccurlyeq$-least element $\<(p_0,\dot a_0),\ldots,(p_n,\dot a_n)>\in Y$ with $\dot a_n=\dot a$ and $p_n\in G$. This is a ranking of $\<X,\vartriangleleft>$ by $\<Y,\preccurlyeq>$, because if $b\vartriangleleft a$ and $\dot b$ is the name we associated with $b$, then we can extend the sequence $\<(p_0,\dot a_0),\ldots,(p_n,\dot a_n)>$ associated with $a$ by adding $(p_{n+1},\dot b)$, where $p_{n+1}$ is a condition in $G$ forcing $\dot b\mathrel{\dot\vartriangleleft}\dot a$ and respecting $\eta$. This is possible because the collection of conditions respecting one more move of $\eta$ is dense below any given condition, and so there is such a condition in $G$, which can then be extended so as to force $\dot b\mathrel{\dot\vartriangleleft}\dot a$. So the $\preccurlyeq$-least element of $Y$ ending in $\dot b$ must be smaller than the corresponding associated element for $a$, and so this is a ranking.
\end{proof}

\section{Preservation of open class determinacy by forcing}

In this section, we shall prove that the principle of determinacy for open class games is preserved by pre-tame class forcing. Let's warm up by showing first that every new class well-order relation added by such forcing is ranked by a ground model class well-order.

\begin{theorem}\label{Theorem.Open-Det-new-wf-relations-are-ranked}
 Assume $\GBC^-$ and the principle of open class determinacy. Then after any pre-tame class forcing, every new class well-founded relation is ranked by a class well-order relation of the ground model.
\end{theorem}

\begin{proof}
Assume $\GBC^-$ and the principle of open class determinacy in the ground model $V$. Suppose that $G\of\P$ is $V$-generic for pre-tame class forcing $\P$. Consider any $\P$-name for a class relation $\dot\vartriangleleft$ on a class $\dot X$, and assume that $\one\forces\dot\vartriangleleft$ is a well-founded relation on $\dot X$. Let $\vartriangleleft=\dot\vartriangleleft_G$ be the actual well-founded relation on $X=\dot X_G$ arising in $V[G]$ from these names.

Consider the following two-player game in the ground model.
$$\begin{array}{rccccccccccc}
{\rm I}\quad   & q_0   &                   & q_1   &                   & q_2   & \\
{\rm II}\quad  &       & (p_0,\dot a_0)    &       & (p_1,\dot a_1)    &       & (p_2,\dot a_2)\quad\raise 8pt\hbox{$\cdots$}\\
\end{array}$$
Player I plays conditions $q_n\in \P$ and player II plays pairs $(p_n,\dot a_n)$, with $p_n\in\P$ and $p_n\forces a_n\in \dot X$. We require that the conditions descend during play $q_0\geq p_0\geq q_1\geq p_1$ and so on and that $p_{n+1}\forces\dot a_{n+1}\mathrel{\dot\vartriangleleft}\dot a_n$. Player I wins if player II cannot play, and otherwise player II wins. So this game is open for player I.

We claim that player II can have no winning strategy for this game in $V$. To see this, suppose toward contradiction that $\sigma$ is a winning strategy for player II. We claim that we can find a play of the game in $V[G]$ that accords with $\sigma$, where all the conditions come from $G$. This will contradict our assumption that $\vartriangleleft$ is well-founded in $V[G]$, since the sequence $\<a_n\mid n\in\omega>$, where $a_n=(\dot a_n)_G$, will be $\vartriangleleft$-descending. The thing to notice is that the collection of conditions $p_0$ played by $\sigma$ in response to the various possible first moves $q_0$ is dense in $\P$, since the game requires $p_0\leq q_0$, and so there must be some $q_0$ with the response $p_0\in G$. More generally, if $p_n\in G$, then the collection of conditions $p_{n+1}$ played by $\sigma$ in response to some possible move $q_n$ is dense below $p_n$, and therefore there is such a move $q_n$ with response $p_{n+1}\in G$. In this way, in $V[G]$ we may construct a play of the game, where all the conditions come from the generic filter $G$ and player II's moves are played in accordance with $\sigma$. Since $p_{n+1}\forces \dot a_{n+1}\mathrel{\dot\vartriangleleft}\dot a_n$, it follows that $\<a_n\mid n\in\omega>$ is a strictly descending sequence in $\vartriangleleft$ in $V[G]$, contradicting our assumption that this relation is well-founded. So player II can have no winning strategy in the game in $V$.

Since we have assumed open class determinacy in $V$, it follows that player I must have a winning strategy $\tau$. Let $T$ be the tree of partial plays of the game that accord with the strategy $\tau$. This is a well-founded class tree in $V$, precisely because every play according to $\tau$ will end with a win for player I in finitely many moves. Place the Kleene-Brouwer order $\preccurlyeq$ on the tree, which is a well-order relation extending the tree order (with the tree growing downward). In other words, we have a rank function on the tree into a well-order $\<T,\preccurlyeq>$ of the ground model.

Let us now find in $V[G]$ an embedding of $\vartriangleleft$ into $\preccurlyeq$, which is a class well-order of the ground model. Associate each $a\in X$ with the smallest $\preccurlyeq$-rank of a position in which $(p,\dot a)$ is played as the last move of that position, for some condition $p\in G$ and some name $\dot a$ with $a=\dot a_G$. The point now is that if $b\vartriangleleft a$ in $V[G]$, then for any position of the game ending with $(p,\dot a)$ for $p\in G$, we may extend it by having player I play any stronger $q\in G$ and then player II respond with any $(p',\dot b)$ where $p'\leq q$ and $p'\in G$ and $p'\forces \dot b\mathrel{\dot\vartriangleleft}\dot a$. Thus, the lowest-rank position in which $\dot b$ appears as the final move must be strictly lower than the corresponding rank of $(p,\dot a)$, since we extended the play with $\dot a$ to a longer play (hence lower in $T$) with $\dot b$. Therefore, we have mapped $b$ strictly below $a$ in the $\preccurlyeq$ order. So this map is a ranking of $\<X,\vartriangleleft>$ by the class well-order $\<T,\preccurlyeq>$ of the ground model, as desired.
\end{proof}

Although the ranking provided in the proof of theorem~\ref{Theorem.Open-Det-new-wf-relations-are-ranked} is not necessarily continuous, nevertheless we shall deduce in corollary~\ref{Corollary.Open-det-no-new-well-orders} that in fact every well-founded relation in $V[G]$ is continuously ranked there by a ground-model class well-order relation.

\begin{theorem}\label{Theorem.Open-determinacy-preserved}
 Assume $\GBC^-$. The principle of open class determinacy is preserved by pre-tame forcing.
\end{theorem}

\begin{proof}
Fix a pre-tame class forcing notion $\P$ and consider any $\P$-name $\dot A$ for an open class game to be played on some class $\dot X$ in the forcing extension. We may regard $\dot A$ as naming the finite sequences corresponding to the basic open sets of the desired payoff class, so that player I wins when the play arrives at a sequence in the class named by $\dot A$. This game is therefore open for player I. We may assume $\one\forces\dot A\of\dot X^{<\omega}$.

Consider the following game $\Gamma_{\dot A}$, to be played in the ground model.
$$\begin{array}{rccccccccccc}
{\rm I}\quad   & (p_0,\dot a_0)   &                   & (p_2,\dot a_2)   &                   & (p_4,\dot a_4)   & \\
{\rm II}\quad  &       & (p_1,\dot a_1)    &       & (p_3,\dot a_3)    &       & (p_5,\dot a_5)\quad\raise 8pt\hbox{$\cdots$}\\
\end{array}$$
Player I plays $(p_0,\dot a_0)$, then II plays $(p_1,\dot a_1)$ and so on, with the conditions descending $p_0\geq p_1\geq p_2$ and so on. Player I wins, if $p_n$ forces that the supplemental play so far $\<\dot a_0,\ldots,\dot a_n>$ is a won position in $\dot A$. Otherwise player II wins.

This is an open game in the ground model, open for player I, and hence it is determined. Suppose first that player I has a winning strategy $\tau$ in the game $\Gamma_{\dot A}$. Let $(p_0,\dot a_0)$ be the initial move played by $\tau$, and let $G\of\P$ be $V$-generic, with $p_0\in G$. We shall now define an associated strategy $\tau^+$ in $V[G]$ for player I in the game $A=\dot A_G$, where only ordinals are played. The strategy $\tau^+$ will proceed by having player I successively imagine conditions $p_n$, with $p_n\in G$, and names $\dot a_n$, so that the actual play $\<a_0,a_1,\ldots,a_n>$ of the game $A$ according to $\tau^+$ is the projection of an imaginary play $\<(p_0,\dot a_0),(p_1,\dot a_1),\ldots,(p_n,\dot a_n)>$ in the game $\Gamma_{\dot A}$ according to $\tau$, using conditions from $G$ and with $(\dot a_k)_G=a_k$. For example, we will achieve this already with the initial move by having $\tau^+$ direct player I to play $(\dot a_0)_G$. Suppose that this imaginary play condition has been achieved up to move $(p_n,\dot a_n)$, with player II to play, and player II plays a move $a_{n+1}$ in the actual game $A$. Pick a name $\dot a_{n+1}$ with $(\dot a_{n+1})_G=a_{n+1}$ and observe that for any condition $p_{n+1}\leq p_n$, if we imagine that player I plays $(p_{n+1},\dot a_{n+1})$ we get a reply $(p_{n+2},\dot a_{n+2})$ from the strategy $\tau$, with $p_{n+2}\leq p_{n+1}$. Since the collection of such resulting conditions $p_{n+2}$ is therefore dense below $p_n$, there must be some condition $p_{n+1}$ for which the reply $p_{n+2}$ is in $G$, and these are the moves in $\Gamma_{\dot A}$ that we now associate with these next two moves in the actual game $A$. Thus, we have defined the strategy $\tau^+$ and achieved the projection-of-$\tau$ condition. Since $\tau$ is winning for player I in $\Gamma_{\dot A}$, it follows that eventually a condition $p_n$ will be played that forces the sequence $\<\dot a_0,\ldots,\dot a_n>$ is in $\dot A$, and since the conditions come from $G$ and those are the names for the actual conditions played in the game $A$, this means that player I will have won the game $A$ by that stage. So $\tau^+$ is a winning strategy for player I in the game $A$ in $V[G]$. Since we assumed $p_0\in G$, this means that $p_0$ forces that player I has a winning strategy for $\dot A$.

Consider now the case where player II has a winning strategy $\sigma$ in the game $\Gamma_{\dot A}$. Let $G\of\P$ be any $V$-generic filter. As in the previous case, we may define a strategy $\sigma^+$ for player II in the game $A$ in $V[G]$, by associating conditions $p_n\in G$ and names $\dot a_n$ so that every play $\<a_0,\ldots,a_n>$ according to $\sigma^+$ is the projection of a play $\<(p_0,\dot a_0),\ldots,(p_n,\dot a_n)>$ according to $\sigma$ in $\Gamma_{\dot A}$, with $p_n\in G$ and $(\dot a_n)_G=a_n$. Suppose that this imaginary-play condition has been achieved this far, and $a_{n+1}$ is the next move played by player I in the actual game $A$. Pick a name $\dot a_{n+1}$ with $(\dot a_{n+1})_G=a_{n+1}$ and observe that for every condition $p_{n+1}\leq p_n$, if we imagine player I playing $(p_{n+1},\dot a_{n+1})$, then there is a reply $(p_{n+2},\dot a_{n+2})$ provided by $\sigma$, and since these replies are therefore dense in $\P$, there must be such conditions in $G$, providing the next two associated plays in $\Gamma_{\dot A}$. The response of strategy $\sigma^+$ to move $a_{n+1}$ is the move $a_{n+2}=(\dot a_{n+2})_G$, which maintains the associated play according to $\sigma$ in $\Gamma_{\dot A}$. We claim that $\sigma^+$ is winning for player II in the game $A$ in $V[G]$. If not, there is some stage of play where the actual play $\<a_0,\ldots,a_n>$ is in $A$. So there is some condition $p_{n+1}\leq p_n$ forcing that the associated names $\<\dot a_0,\ldots,\dot a_n>$ is in $\dot A$, and this would provide a defeating play of $\sigma$ in the game $\Gamma_{\dot A}$, contrary to our assumption that $\sigma$ was winning for player II. So player II has a winning strategy $\sigma^+$ for the game $A$ in $V[G]$.

So we've proved that either there is a condition $p_0$ forcing that player I has a winning strategy in the game $\dot A$ or else every condition forces that player II has a winning strategy. By considering this same fact for the forcing $\P\restrict p$ below an arbitrary condition $p$, we may conclude that there are a dense class of conditions forcing that one or the other player has a winning strategy in $\dot A$, and so $\one$ forces that the game $\dot A$ is determined.
\end{proof}

Note that the argument uses the global choice principle when picking for each move $a$ a name $\dot a$ with $(\dot a)_G=a$. This could be avoided if the game was played on a ground model class $A$, for then one could use check names. For example, if we were playing an open game in $V[G]$ on the class of ordinals $\Ord$, then we wouldn't need global choice.

\begin{corollary}\label{Corollary.Open-det-no-new-well-orders}
 Assume $\GBC^-$ and the principle of open class determinacy. Then after any set forcing or pre-tame class forcing, every new well-founded class relation is continuously ranked by a class well-order of the ground model. In particular, every new class well-order is isomorphic to a class well-order of the ground model.
\end{corollary}

\begin{proof}
If $\GBC^-$ and open class determinacy holds in the ground model, then theorem~\ref{Theorem.Open-Det-new-wf-relations-are-ranked} shows that every well-founded class relation is ranked by a ground-model class well-order relation. Theorem~\ref{Theorem.Open-determinacy-preserved} shows that open class determinacy holds in the forcing extension $V[G]$. By the main result of~\cite{GitmanHamkins2016:OpenDeterminacyForClassGames}, it follows that \ETR\ holds in the extension $V[G]$. Therefore, by theorem~\ref{Theorem.Well-founded-relations-are-ranked} it follows that every well-founded class relation of the extension is continuously ranked by a ground-model class well-order relation. This implies that every new class well-order relation is isomorphic to a ground-model class well-order relation.
\end{proof}

The argument used in the proof of theorem~\ref{Theorem.Open-determinacy-preserved} does not appear to show that clopen determinacy is preserved, for if the game $A$ is clopen, then the game $\Gamma_{\dot A}$ is not in general also clopen, as there could be infinitely long plays with no winner yet; for example, perhaps all the conditions $p_n$ are trivial and don't yet determine the names $\dot a_n$ sufficiently to force an outcome at any particular stage. Since clopen determinacy for class games is equivalent to \ETR, this bears on the question of whether \ETR\ is preserved by forcing.

If we have a ranking to a class well order in the ground model, however, then we could turn it into a clopen game by requiring the players to count down in that ranking also. This is the main idea of theorem~\ref{Theorem.Clopen-determinacy-in-V[G]-for-V-ranked-games}.


\begin{observation}
 If the open player wins an open game in the ground model, then any winning strategy for that player continues to be a winning strategy in any pre-tame forcing extension.
\end{observation}

\begin{proof}
 Suppose that $\tau$ is a winning strategy for the open player in an open game $A$ in the ground model $V$. Let $T$ be the subtree of the game tree consisting of plays in accordance with $\tau$. Since $\tau$ is winning for the open player, this tree is well-founded. It follows that it remains well-founded in any pre-tame forcing extension $V[G]$. And since this continues to the be the tree of plays according to $\tau$ in the extension, it follows that $\tau$ remains winning in the extension.
\end{proof}

\begin{observation} Assume that there is a model of $\GBC+\ETR$. Then there are models $W\of V$ of \GBC\ with the same sets, such that:
 \begin{enumerate}
   \item There is an open game in $W$ that the open player wins in $V$, but not in $W$.
   \item There is an open game in $W$ that the closed player wins in $V$, but not in~$W$.
   \item There is a clopen game in $W$ that is determined in $V$, but not in $W$.
 \end{enumerate}
\end{observation}

\begin{proof}
 The last case implies the first two, since a clopen game is both closed and open, and both players can be seen either as the open player or as the closed player.

 Let $V$ be any model of \GBC+\ETR, and let $W$ be the submodel, with the same first-order part, whose classes consist of those definable from a fixed class well-ordering. This is not a model of \ETR, since there is no first-order truth predicate in that class parameter. So it is not even a model of $\ETR_\omega$. Consider now the counting-down truth-telling game described in~\cite{GitmanHamkins2016:OpenDeterminacyForClassGames}, using that class well-order as a predicate parameter. This is a definable clopen game, which has no winning strategy in $W$, because there is no first-order truth predicate over that class in $W$, but it is determined in $V$, since this is a model of \ETR\ and hence clopen determinacy. The plays and winning conditions of the game are exactly the same in $W$ as in $V$, and while the truth-teller has a winning strategy in $V$, she can have no such strategy in $W$.
\end{proof}

\section{Preservation of \ETR\ by forcing}

Let us now consider the preservation of \ETR\ by forcing. We begin by proving that the principle $\ETR_\Gamma$, for recursions of fixed length $\Gamma$, a class well-order, is preserved by forcing.

\begin{theorem}[Gitman and Hamkins]\label{Theorem.ETR_Gamma-preserved}
 Assume $\GBC^-$. If $\ETR_\Gamma$ holds for a class well-order $\Gamma$, then $\ETR_\Gamma$ continues to hold after any pre-tame class forcing and indeed, after any class forcing by a forcing notion having forcing relations. 
\end{theorem}

\begin{proof}
Suppose that $\ETR_\Gamma$ holds, where $\<\Gamma,\leq_\Gamma>$ is a class well-order of the ground model, and consider any forcing extension $V[G]$ arising from a $V$-generic filter $G$ for a class forcing notion $\P$ having forcing relations in the ground model (this includes all instances of pre-tame class forcing; by the main results of \cite{GitmanHamkinsHolySchlichtWilliams:The-exact-strength-of-the-class-forcing-theorem}, under $\ETRord$ it includes all instances of class forcing). Suppose that $\varphi(x,Z)$ is a first-order formula with parameter $Z=\dot Z_G$, having name $\dot Z$ in the ground model, which we would like to iterate along the class well-order $\Gamma$.

In the ground model, consider the following recursive definition of a set $\bar S\of \Gamma\times V$. We will let $\dot S$ be an associated class name, where $\<\alpha,\dot y>\in \bar S$ just in case $\dot S$ has a name for the pair $\<\check\alpha,\dot y>$. If $\bar S\restrict\gamma$ is defined, for some $\gamma\in\Gamma$, then let $\bar S_\gamma$ be the class of pairs $\<\dot x,p>$ for which $p\forces\varphi(\dot x,\dot S\restrict\gamma,\dot Z)$, where $\dot S\restrict\gamma$ is the corresponding class name for the earlier sections. The point is that this is a recursion of length $\Gamma$ that can be undertaken in the ground model, and so there is indeed a class $\bar S$ and corresponding name $\dot S$ solving this recursion.

It now follows that $\dot S_G$ solves the recursion defined by $\varphi$ in $V[G]$, since by design we have that $\dot x_G\in S_\gamma$ just in case there is some $p\in G$ with $p\forces\varphi(\dot x,\dot S\restrict\gamma,\dot Z)$, which occurs if and only if $V[G]\satisfies\varphi(\dot x_G,\dot S_G\restrict\gamma,Z)$. So $\ETR_\Gamma$ holds in $V[G]$.
\end{proof}

\begin{theorem}\label{Theorem.Clopen-determinacy-in-V[G]-for-V-ranked-games}
  Assume $\GBC^-$. If \ETR\ holds in $V$ and $V[G]$ is a pre-tame class forcing extension, then every clopen class game in $V[G]$, whose game tree is ranked in $V[G]$ by some ground-model class well order, is determined in $V[G]$.
\end{theorem}

\begin{proof}
Consider a clopen game in $V[G]$, which we may view as a well-founded game tree $T$, whose terminal nodes are labeled with the player who wins upon reaching that node, and suppose that $\pi:T\to\Gamma$ is a ranking of this tree in $V[G]$ using a class well order $\Gamma$ in the ground model $V$. Fix names $\dot T$ and $\dot \pi$ for these classes, and consider the \emph{name game}, an associated game in the ground model.
$$\begin{array}{rccccccccccc}
{\rm I}\quad   & (p_0,\dot a_0,\alpha_0)   &                   & (p_2,\dot a_2,\alpha_2)   &                   & (p_4,\dot a_4,\alpha_4)   & \\
{\rm II}\quad  &       & (p_1,\dot a_1)    &       & (p_3,\dot a_3)    &       & (p_5,\dot a_5)\quad\raise 8pt\hbox{$\cdots$}\\
\end{array}$$
The players play a descending sequence of conditions $p_0\geq p_1\geq p_2\geq p_3$ and so on, and each $p_n$ must force that $(\dot a_0,\dot a_1,\ldots,\dot a_n)$ names a valid play in $\dot T$. In addition, for Player I, we require that $p\forces \dot \pi(\dot a_n)=\check\alpha_n$. The game ends if $p_n$ forces that $\dot a_n$ is a terminal node in $\dot T$. In this case, we insist that the play is legal only if $p_n$ also decides the label on that terminal node specifying the winner, and in this case that player is declared also to have won this instance of the name game.

It follows that the $\alpha_n$ must be strictly descending in $\Gamma$, and so the name game will definitely end in finitely many moves. So this is a clopen game, and so by \ETR, the game is determined in $V$.

If player I has a winning strategy $\tau$, then we claim that she can use this winning strategy to create a winning strategy $\tau^+$ in the game $T$ in $V[G]$. As in the proof of theorem~\ref{Theorem.Open-determinacy-preserved}, she will invent conditions $p_n\in G$ for her opponent in such a way that the actual play of the game $T$ according to $\tau^+$ is the projection of a play in the name game according to $\tau$. Thus, she will eventually find herself in the name game with a finite descending sequence of conditions $p_0\geq p_1\geq\cdots\geq p_n$ from $G$ and corresponding names $(\dot a_0,\ldots,\dot a_n)$, which $p_n$ forces are a terminal node of $\dot T$, winning for player I. Thus, the actual play of the game $(a_0, a_1,\ldots,a_n)$ is therefore actually a win for player I in $T$. So $\tau^+$ is a winning strategy in $V[G]$ for player I in~$T$.

If player II has a winning strategy $\sigma$ in the name game, then the argument is similar, except that in this case, player II must invent not only conditions $p_n$ for player I, but also the values $\alpha_n$ in $\Gamma$. Suppose that strategy $\sigma$ has just played $(p_n,\dot a_n)$ for player II in the name game, with $p_n\in G$, and player I plays move $a_{n+1}$ in the actual game $T$. The strategy $\sigma^+$ will select a name $\dot a_{n+1}$ with $(\dot a_{n+1})_G=a_{n+1}$, and then find a condition $p_{n+1}\in G$ and value $\alpha_{n+1}\in\Gamma$ such that $p_{n+1}$ forces $\dot\pi(\dot a_{n+1})=\check\alpha_{n+1}$, and such that furthermore, the reply by $\sigma$ to $(p_{n+1},\dot a_{n+1},\alpha_{n+1})$ is a pair $(p_{n+2},\dot a_{n+2})$ with $p_{n+2}\in G$. This is possible because the collection of such replies is dense below the previous conditon $p_n$, and so there will be such a choice with $p_{n+2}\in G$. Since $\sigma$ is winning for player II in the name game, it follows that eventually a sequence will be played that names a winning state in the actual game $T$, and so $\sigma^+$ is winning for player II in $T$ in~$V[G]$.
\end{proof}

\begin{corollary}
 Assume $\GBC^-$. Then \ETR\ is preserved by countably strategically closed pre-tame class forcing. Consequently, after any such forcing over a model of \ETR, every new class well order is isomorphic to a ground model class well order.
\end{corollary}

\begin{proof}
By theorem~\ref{Theorem.Closed-forcing-every-class-wf-relation-is-ranked}, every well-founded relation in such an extension is ranked by a ground model class well-order. Consequently, by theorem~\ref{Theorem.Clopen-determinacy-in-V[G]-for-V-ranked-games}, every clopen class game in such an extension is determined. By the main theorem of~\cite{GitmanHamkins2016:OpenDeterminacyForClassGames}, this is equivalent to \ETR. Furthermore, by theorem~\ref{Theorem.ETR-implies-continuous-rankings}, the ranking can be refined to continuous rankings, and so every new class well order is continuously ranked by a ground model class well order. Such a ranking is an isomorphism with an initial segment, and so every new class well order is isomorphic to a ground model class well order.
\end{proof}

\printbibliography

\end{document}